\documentclass[12pt]{article}
\usepackage[a4paper, margin=2.3cm]{geometry}
\usepackage{amsmath,amssymb,amsthm}
\usepackage{color}
\usepackage{parskip}

\newtheorem{theorem}{Theorem}[section]

\newtheorem{corollary}[theorem]{Corollary}
\newtheorem{lemma}[theorem]{Lemma}
\newtheorem*{definition*}{Definition}

\newcommand\be{\begin{eqnarray*}}
\newcommand\ee{\end{eqnarray*}}
\newcommand\beq{\begin{equation}}
\newcommand\eeq{\end{equation}}

\newcommand\ben{\begin{eqnarray}}
\newcommand\een{\end{eqnarray}}

\begin{document}
\title{Expansion for the product of matrices in groups
}
\author{
Doowon Koh\thanks{Department of Mathematics, Chungbuk National University. Email: {\tt koh131@chungbuk.ac.kr}}
\and 
    Thang Pham\thanks{Department of Mathematics,  UCSD. Email: {\tt v9pham@ucsd.edu}}
  \and
  Chun-Yen Shen \thanks{Department of Mathematics,  National Taiwan University. Email: {\tt cyshen@math.ntu.edu.tw}}
 \and 
 Le Anh Vinh \thanks{Department of Mathematics, VNU. Email:{\tt vinhla@vnu.edu.vn}} 
  }

\date{}
\maketitle
\date{}
\maketitle
\begin{abstract}
In this paper, we give strong lower bounds on the size of the sets of products of matrices in some certain groups. More precisely, we prove an analogue of a result due to Chapman and Iosevich for matrices in $SL_2(\mathbb{F}_p)$ with restricted entries on a small set. We also provide extensions of some recent results on expansion for cubes in Heisenberg group due to Hegyv\'{a}ri and Hennecart. 
\end{abstract}
\section{Introduction}
Let $\mathbb{F}_p$ be a prime field. We denote by $SL_2(\mathbb{F}_p)$ the set of $2\times 2$ matrices with determinant one over $\mathbb{F}_p$ . Given $A\subset \mathbb{F}_p$, we define 
\[R(A):=\left\lbrace \left(\begin{matrix}
a_{11}&a_{12}\\
a_{21}&a_{22}
\end{matrix}\right)\in SL_2(\mathbb{F}_p)\colon a_{11}, a_{12}, a_{21}\in A\right\rbrace.\]
It was proved by Chapman and Iosevich \cite{Io} by Fourier analytic methods that if $|A|\gg p^{5/6}$ then 
\[|R(A)\cdot R(A)|\gg p^3.\]
Throughout this paper the notation $U \ll V$ means $U \leq cV$ for some absolute constant $c>0$, and $U \gtrsim V$ means $U\gg (\log U)^{-c} V$ for some absolute constant $c>0$.
It has been extensively studied about the size of the products of $R(A)$. In particular, the breakthrough work of H. A. Helfgott \cite{ha} asserts that if $E$ is a subset  of $SL_2(\mathbb {F}_p)$ and is not contained in any proper subgroup with $|E| < p^{3-\delta}$, then $|E \cdot E \cdot E| > c |E|^{1+\epsilon}$ for some $\epsilon=\epsilon(\delta) >0$. The result mentioned above by Chapman and Iosevich is to give a quantitative estimate when the size of the set $A$ is large.  However it is considered to be a difficult problem to obtain some quantitative estimate for the same problem when the size of the set $A$ is not large. It is basically because the Fourier analytic methods are effective only when the size of the set $A$ is large. In this paper, we address the case of small sets, and give a lower bound on the size of $R(A)\cdot R(A)$. Our first result is as follows.
\bigskip
\begin{theorem}\label{thm1}
 Let   $A\subset \mathbb{F}_p.$ If  $|A|\le c p^{\frac{12}{19}}$ for some small constant $c>0$,  then 
\[|R(A)\cdot R(A)|\gg |A|^{\frac{7}{2}+\frac{1}{12}}.\]
\end{theorem}
\bigskip


Let $\mathbb{F}_p$ be a prime field. For an integer $n\ge 1$, the Heisenberg group of degree $n$, denoted by 
$\mathbf{H}_n(\mathbb{F}_p),$ is defined by a set of the following matrices:
\[[\mathbf{x}, \mathbf{y}, z] :=  \begin{bmatrix} 1 & \mathbf{x} & z \\
\mathbf{0} & I_n & \mathbf{y^t} \\
0 & \mathbf{0} & 1
\end{bmatrix}
\]
where $\mathbf{x}, \mathbf{y} \in \mathbb{F}_p^n$, $z \in \mathbb{F}_p$, $\mathbf{y^t}$ denotes the column vector of $\mathbf{y}$, and $I_n$ is the $n \times n$ identity matrix. 
For $A, B, C\subset \mathbb{F}_p$, we define
\[
[A^n, B^n, C]:=\{[\mathbf{x}, \mathbf{y}, z]\colon \mathbf{x}\in A^n, \mathbf{y}\in B^n,\ z\in C\}.
\]
\medskip
A similar question in the setting of the Heisenberg group over prime fields has been recently investigated by Hegyv\'ari and Hennecart in \cite{HHH}, namely, they proved the following theorem.
\bigskip
\begin{theorem}[Hegyv\'ari-Hennecart, \cite{HHH}]
For every $\varepsilon>0,$ there exists a positive integer $n_0=n_0(\epsilon)$ such that if $n\ge n_0$, and $[A^n, B^n, C]\subseteq \mathbf{H}_n(\mathbb{F}_p)$ with \[
|[A^n, B^n, C]|>|\mathbf{H}_n((\mathbb{F}_p))|^{3/4+\varepsilon},
\] 
then there exists a non-trivial subgroup $G$ of $\mathbf{H}_n(\mathbb{F}_p)$ such that $[A^n, B^n, C]\cdot [A^n, B^n, C]$ contains at least $|[A^n, B^n, C]|/p$ cosets of $G$. 
\end{theorem}
In a very recent paper, using results on sum-product estimates, Hegyv\'{a}ri and Hennecart \cite{HH} established some results in the case $n=1$. In particular, they proved that if $A\subset \mathbb{F}_p$ with $|A|\ge p^{1/2}$, then 
\[|[A, A, 0]\cdot [A, A, 0]|\gg \min\left\lbrace p^{1/2}|A|^{5/2}, p^{-1/2}|A|^4  \right\rbrace.\]
In the case when $|A|\le p^{2/3}$, they also showed that
\[|[A, A, 0]\cdot [A, A, 0]|\gg |A|^{\frac{7}{2}}.\]
In this paper, we also extend this result to the setting of Heisenberg group of degree two. For simplicity, we write  $[A^2, A^2, 0]^2$ and $[A^2, A^2, A]^2$ for the products $[A^2, A^2, 0]\cdot [A^2, A^2, 0]$ and $[A^2, A^2, A]\cdot [A^2, A^2, A],$  respectively. We have the following theorems. 
\bigskip
\begin{theorem}\label{thm3}
If $A\subset \mathbb{F}_p$ with $|A|\le p^{\frac{1}{2}}$,  then we have 
\[|[A^2, A^2, 0]^2|\gtrsim |A|^{\frac{11}{2}+\frac{25}{262}}.\]
\end{theorem}
\bigskip
\begin{theorem}\label{thm4} Let $A\subset \mathbb{F}_p$ with $|A|\le p^{\frac{9}{16}}.$ 
 Then we have 
\[|[A^2, A^2, A]^2|\gtrsim |A|^{\frac{11}{2}+\frac{23}{90}}.\]
\end{theorem}
The rest of this paper is organized to provide the complete proofs of our main theorems. More precisely, in Section $2$ we give the proof of Theorem \ref{thm1}, and   in Section $3$ we complete proofs of Theorems \ref{thm3} and \ref{thm4}.
\section{Proof of Theorem \ref{thm1}}
In this section, without loss of generality, we assume that $0\not\in A$.  To prove Theorem \ref{thm1}, we need the following lemmas. 
\bigskip
\begin{lemma}[\cite{pham}, Corollary $3.1$]\label{m-dif1}
Let $X, A\subset \mathbb{F}_p$ with $|X|\ge |A|$. Then we have
\[|X+ A\cdot A|\gg \min\left\lbrace |X|^{1/2}|A|, p\right\rbrace.\]
\end{lemma}
\bigskip
\begin{lemma}
\label{lm2}  Let $A\subset \mathbb{F}_p$  with $|A|\le c p^{2/3}$ for a sufficiently small $c>0.$ Then the number of tuples $(a_1, a_2, a_3, a_4, a_1', a_2', a_3', a_4')\in A^8$ satisfying 
\[a_1a_2+a_3a_4=a_1'a_2'+a_3'a_4'\]
is  $\ll |A|^{13/2}$ 
\end{lemma}
\begin{proof}
For $\lambda, \beta\in \mathbb{F}_p\setminus\{0\}$, one can follow the proof of \cite[Theorem $3$]{RudnevRNSh} to prove that the number of tuples $(a_1, a_2, a_3, a_1', a_2', a_3')\in A^6$ such that 
\[a_1a_2+\lambda a_3=a_1'a_2'+\beta a_3'\]
is $\ll |A|^{9/2}$. 
Thus we see that  for each fixed pair $(a_4, a_4')\in A^2$ the number of tuples $(a_1, a_2, a_3, a_4, a_1', a_2', a_3', a_4')\in A^8$ satisfying 
\[a_1a_2+a_3a_4=a_1'a_2'+a_3'a_4'\]
is  $\ll |A|^{9/2}$. 
Taking the sum over all pairs $(a_4, a_4')\in A^2$, the lemma follows. 

\end{proof}
\bigskip
\begin{lemma}[\cite{frank}, Theorem 4]\label{lienthuoc}
Let $A, B\subset \mathbb{F}_p$ with $|A|\le |B|,$ and let $L$ be a finite set of lines  in $\mathbb{F}_p^2.$ Suppose that
$|A||B|^2\le |L|^3$ and $|A| |L|\ll p^2.$ Then the number of incidences between $A\times B$ and lines in $L$, denoted by $I(A\times B, L)$, satisfies 
\[I(A\times B, L)\ll |A|^{3/4}|B|^{1/2}|L|^{3/4}+|L|.\]
\end{lemma}
\bigskip
The following is an improvement of Lemma $23$ in \cite{s-ds}.
\begin{lemma}\label{lm3}
Let $A, B\subset \mathbb{F}_p$. Then if $|A|=|B|$, and $|A|^2|AB|\ll p^2$,  we have 
\[|A\cap (B+x)|\ll |A|^{-1/2}|AB|^{5/4},\]
for any $x\ne 0$.
\end{lemma}
\begin{proof}
It is clear that 
\[|A\cap (B+x)|\ll \frac{1}{|A||B|}\left\vert \left\lbrace (u, u_*, a, b)\in AB\times AB\times A\times B\colon ub^{-1}-u_*a^{-1}=x \right\rbrace\right\vert.\]
The number of such tuples $(u, u_*, a, b)$ is bounded by the number of incidences between points in $A^{-1}\times AB$ and a set $L$ of lines of the form $b^{-1}Y-u_*X=x$ with $b\in B$ and $u_*\in AB$. Notice that $|A|=|A^{-1}|$ and $|L|=|B||AB|.$ Thus if $|A|=|B|$ and $|A|^2|AB|\ll p^2$, Lemma \ref{lienthuoc} implies that 
\[I(A^{-1}\times AB,~ L)\ll |A|^{3/2}|AB|^{5/4},\]
which completes the proof of the theorem. 
\end{proof}
\bigskip
\begin{lemma}[\cite{m}, Theorem 2]\label{m-dif2}
If  $A \subset \mathbb{F}_p$ with $|A|\le p^{9/16}$, then we have
\[|A\pm A|^{18}|AA|^9 \gtrsim |A|^{32}.\]
\end{lemma}
\bigskip
We are now ready to prove Theorem \ref{thm1}.
\begin{proof}[Proof of Theorem \ref{thm1}]
Without loss of generality, we may assume that $0\notin A.$
Let $M_1$ and $M_2$ be matrices in $R(A)$ presented as follows:
\[M_1:=\left(\begin{matrix}
a_{11}&a_{12}\\
a_{21}&\frac{1+a_{12}a_{21}}{a_{11}}
\end{matrix}\right), ~~M_2:=\left(\begin{matrix}
b_{11}&b_{12}\\
b_{21}&\frac{1+b_{12}b_{21}}{b_{11}}
\end{matrix}\right).\]
Suppose that 
\[M_1\cdot M_2=\left(\begin{matrix}
t&\alpha\\
\beta&\frac{1+\alpha\beta}{t}
\end{matrix}\right),\]
where $t\ne 0$ and $\alpha, \beta\in \mathbb F_p.$
Then we have the following system 
\begin{equation}\label{eq1}
a_{11}b_{11}+a_{12}b_{21}=t,~ \frac{b_{12}t}{b_{11}}+\frac{a_{12}}{b_{11}}=\alpha, ~~\frac{a_{21}t}{a_{11}}+\frac{b_{21}}{a_{11}}=\beta.
\end{equation}
Let us identify the matrix $M_1\cdot M_2$ with $(t, \alpha, \beta) \in \mathbb F_p^*\times \mathbb F_p^2.$
Notice that $R(A)\cdot R(A)$ contains each $(t, \alpha, \beta)\in \mathbb F_p^*\times \mathbb F_p^2$ satisfying the system \eqref{eq1} for some $(a_{11}, a_{12}, a_{21}, b_{11}, b_{12}, b_{21})$ in $A^6.$
Therefore, we aim to estimate a lower bound of the number of $(t, \alpha, \beta)\in \mathbb F_p^*\times \mathbb F_p^2$ such that the system \eqref{eq1} holds for some $(a_{11}, a_{12}, a_{21}, b_{11}, b_{12}, b_{21})\in A^6.$ To this end, let $\epsilon>0$ be a parameter chosen later. We now consider two following cases:
\begin{enumerate}
\item If $|AA|\ge |A|^{1+\epsilon}$, then it follows from Lemma \ref{m-dif1} that 
\[|AA+AA|\gg \min\{|A|^{\frac{3}{2}+\frac{\epsilon}{2}},~p\}=|A|^{\frac{3}{2}+\frac{\epsilon}{2}},\]
where we assume that  
\begin{equation*}|A|\le p^{\frac{2}{3+\epsilon}}.\end{equation*}
From the system (\ref{eq1}) and the above fact, we obtain that if $|A|\le p^{\frac{2}{3+\epsilon}}$ and $|AA|\ge |A|^{1+\epsilon}$, then
\begin{equation}\label{eq3.5}|R(A)\cdot R(A)|\gg |AA+AA||A|^2\gg |A|^{\frac{7}{2}+\frac{\epsilon}{2}},\end{equation}
where the first $\gg$ follows, because in the system \eqref{eq1}, for each non-zero $t\in AA+AA$, if we fix a quadruple $(a_{11}, b_{11}, a_{12}, b_{21})\in A^4$ with $a_{11}b_{11}+a_{12}b_{21}=t$, then $\alpha, \beta$ are determined in terms of $b_{12}\in A$ and $a_{21}\in A$, respectively. 

\item If $|AA|\le |A|^{1+\epsilon}$, then we consider as follows. For $t, \alpha, \beta\in \mathbb{F}_p$, let $\nu(t, \alpha, \beta)$ be the number of solutions $(a_{11}, a_{12}, a_{21}, b_{11}, b_{12}, b_{21})$ of the system (\ref{eq1}). For the case $t=0$, we have 
\[\sum_{\alpha, \beta}\nu(0, \alpha, \beta)\le |A|^5.\]
Indeed, for each choice of $(b_{11}, a_{12}, b_{21})\in A^3$, $a_{11}$ is determined uniquely, and $\alpha, \beta$ are determined. In addition, $a_{21}$ and $b_{12}$ can be taken as arbitrary elements of $A.$

By the Cauchy-Schwarz inequality, we have 
\[(|A|^6-|A|^5)^2\le \left(\sum_{t\ne 0, \alpha, \beta} \nu(t, \alpha, \beta)\right)^2\le |R(A)\cdot R(A)|\sum_{t\ne 0, \alpha, \beta} \nu^2(t, \alpha, \beta).\]
This implies that 
\begin{equation}\label{CSd}|R(A)\cdot R(A)|\gg \frac{|A|^{12}}{T},\end{equation}
where $\displaystyle T:=\sum_{t\ne 0, \alpha, \beta} \nu^2(t, \alpha, \beta)$. 

In the next step, we are going to show that 
\[T\ll |A|^{8+\frac{5\epsilon}{2}}.\]

To see this,  observe by definition of $\nu(t,\alpha, \beta)$ that for each $(t, \alpha, \beta)\in \mathbb F_p^*\times \mathbb F_p^2,$  the value $\nu^2(t, \alpha, \beta)$ is the number of 12-tuples $(a_{11}, a_{12}, a_{21}, b_{11}, b_{12}, b_{21}, a'_{11}, a'_{12}, a'_{21}, b'_{11}, b'_{12}, b'_{21})\in A^{12}$ satisfying the following:
\begin{align*} a_{11}b_{11}+a_{12}b_{21}&=t= a'_{11}b'_{11}+a'_{12}b'_{21}\\
   \frac{b_{12}t}{b_{11}} + \frac{a_{12}}{b_{11}}&=\alpha=\frac{b'_{12}t}{b'_{11}} + \frac{a'_{12}}{b'_{11}}\\
   \frac{a_{21}t}{a_{11}} + \frac{b_{21}}{a_{11}}&=\beta=\frac{a'_{21}t}{a'_{11}} + \frac{b'_{21}}{a'_{11}}.\end{align*}
Thus the value of $\displaystyle T=\sum_{t\ne 0, \alpha, \beta} \nu^2(t, \alpha, \beta)$ can be written by
$ \displaystyle \sum_{t\ne 0} \Omega(t)$
where $\Omega(t)$ denotes the number of 12-tuples $(a_{11}, a_{12}, a_{21}, b_{11}, b_{12}, b_{21}, a'_{11}, a'_{12}, a'_{21}, b'_{11}, b'_{12}, b'_{21})\in A^{12}$ satisfying the following:
\begin{align}\label{Aeq1} a_{11}b_{11}+a_{12}b_{21}&=t= a'_{11}b'_{11}+a'_{12}b'_{21}\\
\label{Aeq2}\frac{b_{12}t}{b_{11}} + \frac{a_{12}}{b_{11}}&=\frac{b'_{12}t}{b'_{11}} + \frac{a'_{12}}{b'_{11}}\\
\label{Aeq3}  \frac{a_{21}t}{a_{11}} + \frac{b_{21}}{a_{11}}&=\frac{a'_{21}t}{a'_{11}} + \frac{b'_{21}}{a'_{11}}.\end{align}
Now notice that  Lemma \ref{lm2} implies that if $|A|\ll p^{2/3},$ then there are at most $|A|^{13/2}$\\ 
$8$-tuples $(a_{11}, b_{11}, a_{12}, b_{21}, a'_{11}, b'_{11}, a'_{12}, b'_{21})$ in $A^8$ satisfying the equations \eqref{Aeq1} 
for some $t\ne 0.$   One can also check that among these tuples, there are at most $|A|^6$ $(\leq \frac{|A|^{13/2}}{2})$ tuples with  $a'_{12}b'^{-1}_{11}-a_{12}b^{-1}_{11}=0$. Hence, without loss of generality, we may assume that all tuples satisfy $a'_{12}b'^{-1}_{11}-a_{12}b^{-1}_{11}\ne 0$.

For such a fixed $8$-tuple $(a_{11}, b_{11}, a_{12}, b_{21}, a'_{11}, b'_{11}, a'_{12}, b'_{21})\in A^8$, we now deal with the equation \eqref{Aeq2} which can be rewritten by
\begin{equation}\label{eqq1}
\frac{b_{12}}{t^{-1}b_{11}}+a_{12}b^{-1}_{11}=\frac{b'_{12}}{t^{-1}b'_{11}}+a'_{12}b'^{-1}_{11}.\end{equation}
Set $Q=\frac{t}{b_{11}}\cdot A$,~ $Q'=\frac{t}{b'_{11}}\cdot A$, and $x=a_{12}b^{-1}_{11}-a'_{12}b'^{-1}_{11}\ne 0$.  Then the number of solutions $(b_{12},b_{12}')\in A^2$ of (\ref{eqq1}) is the size of $Q\cap (Q'-x)$. It is clear that $|Q|=|Q'|=|A|,$ because $t\ne 0$ and we have assumed that $0\notin A$ so that $t/b_{11}, t/b'_{11}\ne 0.$ We also see that 
$$ |Q|^2|Q\cdot Q'|=|A|^2|AA| \le |A|^{3+\epsilon},$$
where we used the assumption that $|AA|\le |A|^{1+\epsilon}.$ 
Applying Lemma \ref{lm3}, we obtain that if $|A| \ll p^{2/(3+\epsilon)},$ then 
\[|Q\cap (Q'-x)| \ll |A|^{\frac{3}{4}+\frac{5\epsilon}{4}}.\]
The same argument works identically for the equation \eqref{Aeq3} which can be restated by
\begin{equation}\label{eqq1'}
\frac{a_{21}}{t^{-1}a_{11}}+b_{21}a^{-1}_{11}=\frac{a'_{21}}{t^{-1}a'_{11}}+b'_{21}a'^{-1}_{11}.\end{equation}
In short, we have proved that if $|A|\ll p^{2/(3+\epsilon)}$ and $|AA|\le |A|^{1+\epsilon}$, then 
\[T\ll |A|^{\frac{13}{2}}|A|^{\frac{3}{4}+\frac{5\epsilon}{4}}|A|^{\frac{3}{4}+\frac{5\epsilon}{4}}= |A|^{8+\frac{5\epsilon}{2}}.\]
Therefore, combining \eqref{CSd} with this estimate  yields that if $|A|\ll p^{2/(3+\epsilon)}$ and $|AA|\le |A|^{1+\epsilon}$, then
\begin{equation}\label{eq3}|R(A)\cdot R(A)|\gg |A|^{4-\frac{5\epsilon}{2}}.\end{equation}
\end{enumerate}
Finally, if we choose $\epsilon= 1/6$, then it follows from (\ref{eq3.5}) and (\ref{eq3}) that if $|A|\ll p^{12/19}$, then
\[|R(A)\cdot R(A)| \gg |A|^{\frac{7}{2}+\frac{1}{12}},\]
which completes the proof of Theorem \ref{thm1}. 
\end{proof}
\bigskip
In the case of arbitrary finite fields, we have the following result.
\bigskip
\begin{theorem}\label{thm2}
Let $q = p^n$ and let $A$ be a subset of $\mathbb{F}^{*}_{q}$. If $|A \cap \lambda F| \leq |F|^{1/2}$ for any proper subfield $F$ of $\mathbb{F}_{q}$ and any $\lambda \in \mathbb{F}_q$, then we have 
\[|R(A)\cdot R(A)|\gtrsim |A|^{3+\frac{1}{5}}.\]
\end{theorem}
\bigskip
To prove Theorem \ref{thm2} we make use of the following results.
\bigskip
\begin{theorem}[\cite{lili}] \label{lli}
With the assumptions of Theoem \ref{thm2}, we have
\[
	\max\{|A+A|, |AA| \} \gtrsim |A|^{12/11}.
\]
\end{theorem}

\begin{theorem}\label{thm-oliver}
For $A, B\subset \mathbb{F}_q$, suppose that $|A\cap \lambda F|, |B\cap \lambda F|\le |F|^{1/2}$ for any subfield $F$ of $\mathbb{F}_q$ and any $\lambda \in \mathbb{F}_q$. Then we have 
\[|A+AB|\gg \min \{|A||B|^{1/5}, |A|^{3/4}|B|^{2/4}\}.\]
\end{theorem}
\begin{corollary}\label{coxt1}
For $A\subset \mathbb{F}_q$, suppose that $|A\cap \lambda F|\le  |F|^{1/2}$ for any subfield $F$ of $\mathbb{F}_q$ and any $\lambda \in \mathbb{F}_q$. Then we have 
\[|AA+AA|\gg |A|^{6/5}.\]
\end{corollary}

\begin{proof}
Given a nonzero $x \in A$, we have
$|AA+AA|= |\frac{A}{x} \frac{A}{x}+\frac{A}{x} \frac{A}{x}| \geq |\frac{A}{x}\frac{A}{x}+\frac{A}{x}| \gg |A|^{6/5}$ by Theorem 2.8. 

\end{proof}

\bigskip
we are now ready to prove Theorem \ref{thm2}.
\begin{proof}[Proof of Theorem \ref{thm2}]
Recall from \eqref{eq3.5} that  
\[|R(A)\cdot R(A)|\gg |AA+AA||A|^2.\]
Thus the Theorem follows directly from  Corollary \ref{coxt1}.
\end{proof}
In the rest of this section, we present the proof of Theorem \ref{thm-oliver}, for which the authors communicated with Oliver Roche-Newton.
\subsection{Proof of Theorem \ref{thm-oliver}}
To prove Theorem \ref{thm-oliver}, we make use of the following lemmas. 

The first lemma is the Pl\"{u}nnecke-Ruzsa inequality. 
\begin{lemma}[\cite{ttr}, Theorems 1.6.1, 1.81]\label{lm-oliver1}
Let $X, B_1, \ldots, B_k$ be subsets of $\mathbb{F}_q$. Then we have 
\[|B_1+\cdots +B_k|\le \frac{|X+B_1|\cdots |X+B_k|}{|X|^{k-1}},\]
and 
\[|B_1-B_2|\le \frac{|X+B_1||X+B_2|}{|X|}.\]
\end{lemma}
One can modify the proof of Corollary $1.5$ in \cite{shen} due to Katz and Shen to obtain the following. 
\bigskip
\begin{lemma}\label{lm-oliver2}
Let $X, B_1, \ldots, B_k$ be subsets in $\mathbb{F}_q$. Then, for any $0<\epsilon<1$, there exists a subset $X'\subset X$ such that $|X'|\ge (1-\epsilon)|X|$ and 
\[|X'+B_1+\cdots+B_k|\le c\cdot \frac{|X+B_1|\cdots|X+B_k|}{|X|^{k-1}},\]
for some positive constant $c=c(\epsilon).$
\end{lemma}
We also have the following lemma from \cite{lili}.
\bigskip
\begin{lemma}\label{lm-oliver3}
Let $B$ be a subset of $\mathbb{F}_q$ with at least two elements, and define $\mathbb{F}_B$ as the subfield generated by $B$. Then there exists a polynomial $P(x_1, \ldots, x_n)$ of several variables with coefficients belonging to the prime field $\mathbb{F}_p$ such that 
\[P(B, \ldots, B)=\mathbb{F}_B.\]
\end{lemma}
\bigskip
We are now ready to prove Theorem \ref{thm-oliver}. 
\bigskip
\begin{proof}[Proof of Theorem \ref{thm-oliver}]
WLOG, we may assume $1 \in A$ by considering $\frac{1}{x} A$ for some $x \in A$. We first define the ratio set: 
\[R(A, B):=\left\lbrace  \frac{a_1-a_2}{b_1-b_2}\colon a_1, a_2\in A, b_1, b_2\in B, b_1\ne b_2\right\rbrace.\]
We now consider the following cases:

{\bf Case 1:} $1+R(A, B)\not\subset R(A, B)$.

In this case, there exist $a_1, a_2\in A$ and $b_1\ne b_2\in B$ such that 
\[r:=1+\frac{a_1-a_2}{b_1-b_2}\not\in R(A, B).\]

First, we apply Lemma \ref{lm-oliver2} so that there exists a subset $A'\subset A$ such that $|A'|\gg |A|$ and 
\begin{equation}\label{eq-o-1}
|(b_1-b_2)A'+(b_1-b_2)B+(a_1-a_2)B|\ll \frac{|A+B||(b_1-b_2)A+(a_1-a_2)B|}{|A|}.\end{equation}
On the other hand, we have 
\begin{equation}\label{eq-o-2}|(b_1-b_2)A'+(b_1-b_2)B+(a_1-a_2)B| \geq |A'+rB|.\end{equation}
Since $r\not\in R(A, B)$, the equation 
\[a_1+rb_1=a_2+rb_2\]
has no non-trivial solutions, i.e. solutions $(a_1, b_1, a_2, b_2)$ with $b_1\ne b_2$. 
This implies that 
\begin{equation}\label{eq-o-3}|A'+rB|=|A'||B|. \end{equation}
We now give an upper bound for $(b_1-b_2)A+(a_1-a_2)B=b_1A+a_1B-b_2A-a_2B$ which will be used in the rest of the proof. 

Lemma \ref{lm-oliver2} tells us that there exists a subset $X\subset A$ such that $|X|\gg |A|$ and 
\[|X+b_1A+a_1B|\ll \frac{|A+b_1A||A+a_1B|}{|A|}\ll \frac{|A+AB|^2}{|A|},\]
and there exists a subset $X'\subset X$ with $|X'|\gg |X|$ such that 
\[|X'+b_2A+a_2B|\le \frac{|X+b_2A||X+a_2B|}{|X|}\ll \frac{|A+AB|^2}{|A|}. \]
Applying Lemma \ref{lm-oliver1}, we have 
\begin{align}\label{eq-o-4}
|b_1A+a_1B-b_2A-a_2B|&\le \frac{|X'+b_1A+a_1B||X'+b_2A+a_2B|}{|X'|}\\
&\ll \frac{|X+b_1A+a_1B||X'+b_2A+a_2B|}{|A|}\nonumber\\
&\le \frac{|A+BA|^4}{|A|^3}\nonumber.
\end{align}
Putting (\ref{eq-o-1}-\ref{eq-o-4}) together, we obtain 
\[|A+AB|^5\gg |A|^5|B|,\]
and we are done.

{\bf Case 2:} $B\cdot R(A, B)\not\subset R(A, B)$.
Similarly, in this case, there exist $a_1, a_2\in A$ and $b, b_1, b_2\in B$ such that 
\[r:=b\cdot \frac{a_1-a_2}{b_1-b_2}\not\in R(A, B).\]
Since $0\in R(A, B)$, we have $b\ne 0$, and $a_1\ne a_2$. This gives us that $r^{-1}$ exists.

Using the same argument as above, we have 
\begin{align*}
|A||B|=|A+rB|&=|r^{-1}A+B|\le \frac{|b^{-1}A+A||(a_1-a_2)B+(b_1-b_2)A|}{|A|}\\
&\le \frac{|A+AB||A+AB|^4}{|A|^4}.
\end{align*}
Thus we obtain 
\[|A+AB|^5\gg |A|^5|B|,\]
and we done.

{\bf Case 3:} $B^{-1}\cdot R(A, B)\not\subset R(A, B)$.

As above, in this case, there exist $a_1, a_2\in A$ and $b\ne 0, b_1\ne b_2\in B$ such that 
\[r:=b^{-1}\cdot \frac{a_1-a_2}{b_1-b_2}\not\in R(A, B).\]
Since $0\in R(A, B)$, we have $a_1\ne a_2$. This gives us that $r^{-1}$ exists. The rest is the same as the Case $2$.

{\bf Case 4:} We consider the case when
\begin{align}
1+R(A, B)&\subset R(A, B)\\
B\cdot R(A, B)&\subset R(A, B)\\
B^{-1}\cdot R(A, B)&\subset R(A, B).
\end{align}

Now we are going to show that for any polynomial $P(x_1, \ldots, x_n)$ in $n$ variables, for some positive integer $n$, and coefficients belonging $\mathbb{F}_p$ such that 
\[P(B, \ldots, B)+R(A, B)\subset R(A, B).\]

Indeed, it is enough to show that 
\[1+R(A, B)\subset R(A, B), ~B^d+R(A, B)\subset R(A, B),\]
for any integer $d\ge 1$, and $B^d=B\cdots B$ ($d$ times).

It follows from the assumption that  the first condition $1+R(A, B)\subset R(A, B)$ is satisfied.   For the second condition, it is sufficient to prove it for $d=2$, since we can use inductive arguments. 

Let $b, b'$ be arbitrary elements in $B$. We now show that 
\[bb'+R(A, B)\subset R(A, B).\]
If either $b=0$ or $b'=0$, then we are done. Thus we may assume that $b\ne 0$ and $b'\ne 0$. 

First we have 
\[b+R(A, B)=b(1+b^{-1}R(A, B))\subset b(1+R(A, B)))\subset R(A, B),\]
and 
\[bb'+R(A, B)=b(b'+b^{-1}R(A, B))\subset b(b'+R(A, B))\subset bR(A, B)\subset R(A, B).\]

In other words, for any polynomial $P(x_1, x_2, \ldots, x_n)\in \mathbb{F}_p[x_1, \ldots, x_n]$ we have 
\[P(B, \ldots, B)+R(A, B)\subset R(A, B).\]

Furthermore, Lemma \ref{lm-oliver3} tells us that there exists a polynomial $P$ such that $P(B, \ldots, B)=\mathbb{F}_B$. 

This implies that 
\[\mathbb{F}_B+R(A, B)\subset R(A, B).\]

It follows from our assumption of the theorem that 
\[|B|=|B\cap \mathbb{F}_B|\le |\mathbb{F}_B|^{1/2}.\]
Hence, $|R(A, B)|\ge |\mathbb{F}_B|\ge |B|^2$.

Next we shall show that there exists $r\in R(A, B)$ such that either 
\[|A+rB|\gg |A||B|,\]
or 
\[|A+rB|\gg |B|^2.\]
Indeed, let $E^+(X, Y)$ be the number of tuples $(x_1, x_2, y_1, y_2)\in X^2\times Y^2 $ such that 
\[x_1+y_1=x_2+y_2.\]
We have the sum $\sum_{r\in R(A, B)}E^+(A, rB)$ is the number of tuples $(a_1, a_2, b_1, b_2)\in A^2\times B^2$ such that 
\[a_1+rb_1=a_2+rb_2\]
with $a_1, a_2\in A$, $b_1, b_2\in B$ and $r\in R(A, B)$. It is easy to see that there are at most $|R(A, B)||A||B|$ tuples with $a_1=a_2, b_1=b_2$, and at most $|A|^2|B|^2$ tuples with $b_1\ne b_2$. Therefore, we get
\[\sum_{r\in R(A, B)}E^+(A, rB)\le |R(A, B)||A||B|+|A|^2|B|^2\le |R(A, B)|(|A||B|+|A|^2).\]
By the pigeon-hole principle, there exists $r:=\frac{a_1-a_2}{b_1-b_2}\in R(A, B)$ such that 
\[E(A, rB)\le |A||B|+|A|^2.\]
So, either 
\[|A+rB|\gg |A||B|,\]
or 
\[|A+rB|\gg |B|^2.\]

We now fall into two small cases:
\begin{itemize}
\item[1.] If $|A+rB|\gg |A||B|$, then, applying Lemma \ref{lm-oliver1}, we have 
\[|A||B|=|A+rB|= |(b_1-b_2)A+(a_1-a_2)B|\le \frac{|AB+A|^4}{|A|^3},\]
which gives us 
\[|A+AB|\gg |A||B|^{1/2}.\] 
\item[2.] If $|A+rB|\gg |B|^2$, then we have 
\[|B|^2\ll |A+rB|= |(b_1-b_2)A+(a_1-a_2)B|\le \frac{|AB+A|^4}{|A|^3},\]
which gives us 
\[|A+AB|\gg |A|^{3/4}|B|^{2/4}.\]
This completes the proof of the theorem.
\end{itemize}
\end{proof}
\section{Proofs of Theorems \ref{thm3} and \ref{thm4}}
In the proof of Theorem \ref{thm3}, we make use of the following version of Balog-Szemer\'{e}di-Gowers theorem due to Schoen \cite{schoen}. 
\bigskip
\begin{theorem}[\cite{schoen}, Theorem 1.1]\label{bsg}
Let $G$ be an abelian group.  Suppose that $A$ is a subset of $G$, and $E^+(A)$ denotes the additive energy which is the number of solutions $(a,b,c,d)\in A^4$ to the equation $a+b=c+d.$
If $E^+(A)$ is equal to $k|A|^3$, then there exists $A'\subset A$ with $|A'|\gg k|A|$ such that 
\[|A'-A'|\ll k^{-4}|A'|.\]
\end{theorem}
\bigskip
We also will need the following results. 
\bigskip
\begin{theorem}\label{thm-long}
For $A, B, C, D\subset \mathbb{F}_p$, let $Q(A, B, C, D)$ be the number of $8$-tuples \[(a_1, b_1, c_1, d_1, a_2, b_2, c_2, d_2)\in (A\times B\times C\times D)^2\]
such that 
\[a_1b_1+c_1d_1=a_2b_2+c_2d_2.\]
We have 
\[Q(A, B, C, D)\lesssim \frac{|A|^2|B|^2|C|^2|D|^2}{p}+|C|^2|B||D|^{3/2}|A|^{1/2}E^\times(A, B)^{1/2}+|A||D|^3|B||C|+|A|^3|D||B||C|,\]
where 
\[E^\times (A, B)=\#\{(a_1, a_2, b_1, b_2)\in A^2\times B^2\colon a_1b_1=a_2b_2\}.\]
\end{theorem}
To prove this theorem, we need the following version of point-plane incidence bound due to Rudnev in \cite{m}.
\bigskip
\begin{theorem}[\cite{m}]\label{point-plane}
Let $P$ be a set of points in $\mathbb{F}_p^3$ and let $\Pi$ be a set of planes in $\mathbb{F}_p^3$. Suppose that $|P|\le |\Pi|$, and there are at most $k$ collinear points in $P$ for some $k$, then the number of incidences between $P$ and $\Pi$ is bounded by 
\[I(P, \Pi)\le \frac{|P||\Pi|}{p}+|P|^{1/2}|\Pi|+k|P|.\]
\end{theorem}
\bigskip 

we are now ready to prove Theorem \ref{thm-long}. We will follow the ideas of the proof of \cite[Theorem $32$]{s-new}.

\bigskip
\begin{proof}[Proof of Theorem \ref{thm-long}]
We have 
\[Q(A, B, C, D)=\sum_{\lambda, \mu}r_{CD}(\lambda)r_{AB}(\mu)n(\lambda, \mu),\]
where $r_{CD}(\lambda)$ is the number of pairs $(c, d)\in C\times D$ such that $cd=\lambda$, $r_{AB}(\mu)$ is the number of pairs $(a, b)\in A\times B$ such that $ab=\mu$, 
and $n(\lambda, \mu)=\sum_{x}r_{AB+\lambda}(x)r_{CD+\mu}(x)$. If we split the sum $Q(A, B, C, D)$ into intervals, we get
\[Q(A, B, C, D)\ll \sum_{i=1}^{L_1} \sum_{j=1}^{L_2}\sum_{\lambda, \mu}n(\lambda, \mu) r^{(i)}_{CD}(\lambda)r^{(j)}_{AB}(\mu),\]
where $L _1\le \log(|C||D|), L_2\le \log(|A||B|)$, $r_{AB}^{(i)}(\mu)$ is the restriction of the function $r_{AB}(x)$ on the set $P_i:=\{\mu \colon \Delta_i\le r_{AB}(\mu)< 2\Delta_i\}$, and $r_{CD}^{(i)}(\lambda)$ is the restriction of the function $r_{CD}(x)$ on the set $P_i:=\{\lambda\colon \Delta_i\le r_{CD}(\lambda)< 2\Delta_i\}$. Applying the pigeon-hole principle two times, there exist sets $P_i$ and $P_j$ such that 
\[Q(A, B, C, D)\lesssim \sum_{\lambda, \mu}n(\lambda, \mu)r_{CD}^{(i)}(\lambda)r_{AB}^{(j)}(\mu)\lesssim \Delta_i\Delta_j\sum_{\lambda, \mu}n(\lambda, \mu)P_i(\mu)P_j(\lambda),\]
where $P_i(x)$ is the indicator function of the set $P_i$.  For the simplicity, we suppose that $i=1$ and $j=2$. 

One can check that the sum $\sum_{\lambda, \mu}n(\lambda, \mu)P_1(\lambda)P_2(\mu)$ is the number of incidences between points $(a, d, \lambda)\in A\times D\times P_1 \subset \mathbb{F}_p^3$ and planes in $\mathbb{F}_p^3$ defined by 
\[bX-cY+Z=\mu,\]
where $b\in B, c\in C, \mu\in P_2$. 

With the way we define the plane set, it follows from \cite{frank1} that we can apply Theorem \ref{point-plane} with $k=\max\{|A|, |D|\}$. Thus we obtain 
\begin{align}\label{com-pli}Q(A, B, C, D)&\lesssim \Delta_1\Delta_2\left(\frac{|A||B||C||D||P_1||P_2|}{p}\right)\\&+\Delta_1\Delta_2\left(|A|^{1/2}|B||C||D|^{1/2}|P_1|^{1/2}|P_2|+\max\{|A|, |D|\}|A||D||P_1|\right).\nonumber
\end{align}

It is clear in our argument that we can switch the point set and the plane set, we also can do the same thing for $P_1$ and $P_2$ in the definition of the point set and the plane set. So, without loss of generality, we can assume that $|P_1|\le |P_2|$, $|A||D|\le |B||C|$. We now consider the following cases:
\begin{itemize}
\item If the second term dominates, then we have 
\[Q(A, B, C, D)\lesssim |C|^2|B||D|^{3/2}|A|^{1/2}E^\times(A, B)^{1/2},\]
since 
\[\Delta_2|P_2|\le |C||D|, ~\Delta_1|P_1|^{1/2}\le E^\times (A, B)^{1/2}. \]

\item If the first term dominates, then we have  
\[Q(A, B, C, D)\lesssim \frac{|A|^2|B|^2|C|^2|D|^2}{p}.\]
since 
\[\Delta_2|P_2|\le |C||D|, ~\Delta_1|P_1| \le |A||B|. \]
\item If the last term dominates, then we study the following:
\begin{enumerate}
\item Suppose $|A|\le |D|$. If $|D|\le |P_2|$, then it is easy to check that the second term  in (\ref{com-pli}) will be bigger than the last term. Thus, we can suppose that $|D|\ge |P_2|$. Since $|P_1|\le |P_2|$, we have 
$$|A||D|^2|P_1|\le |A||D|^3.$$ On the other hand, it is clear that $\Delta_1\Delta_2\le |B||C|$. This means
\[Q(A, B, C, D)\lesssim |A||D|^3|B||C|.\]
\item Suppose $|A|\ge |D|$. By repeating the same argument, we obtain 
\[Q(A, B, C, D)\lesssim |A|^3|D||B||C|.\]
\end{enumerate}
\end{itemize}
This completes the proof of the theorem.
\end{proof}
\bigskip
\begin{proof}[Proof of Theorem \ref{thm3}]
Let $N$ be the number of tuples \[(x_1, y_1, z_1, t_1, x_2, y_2, z_2, t_2, x'_1, y'_1, z'_1, t'_1, x'_2, y'_2, z'_2, t'_2)\in A^{16}\]
such that $
[\mathbf{x}, \mathbf{y}, 0]\cdot [\mathbf{z}, \mathbf{t}, 0]=[\mathbf{x'}, \mathbf{y'}, 0]\cdot [\mathbf{z'}, \mathbf{t'}, 0]$. This can be expressed as follows:
\begin{equation}\label{eqqq1}\left(\begin{matrix}
1& x_1&x_2&0\\
0&1&0&y_1\\
0&0&1&y_2\\
0&0&0&1\\
\end{matrix}\right)\cdot \left(\begin{matrix}
1& z_1&z_2&0\\
0&1&0&t_1\\
0&0&1&t_2\\
0&0&0&1\\
\end{matrix}\right)=\left(\begin{matrix}
1& x'_1&x'_2&0\\
0&1&0&y'_1\\
0&0&1&y'_2\\
0&0&0&1\\
\end{matrix}\right)\cdot \left(\begin{matrix}
1& z'_1&z'_2&0\\
0&1&0&t'_1\\
0&0&1&t'_2\\
0&0&0&1\\
\end{matrix}\right).\end{equation}
Thus by the Cauchy-Schwarz inequality, we have 
\begin{equation}\label{FormulaA2A20}|[A^2, A^2, 0]^2|\ge \frac{|A|^{16}}{N}.\end{equation}
From (\ref{eqqq1}), observe that  $N$ is the number of tuples
$$(x_1, y_1, z_1, t_1, x_2, y_2, z_2, t_2, x'_1, y'_1, z'_1, t'_1, x'_2, y'_2, z'_2, t'_2)\in A^{16}$$ 
satisfying the following system:
\begin{align}
\label{Nsystem1} x_1+z_1=x'_1+z'_1,&\quad x_2+z_2=x'_2+z'_2\\
\label{Nsystem2} y_1+t_1=y'_1+t'_1,&\quad y_2+t_2= y'_2+t'_2\\
\label{Nsystem3} x_1t_1+x_2t_2&=x'_1t'_1+x'_2t'_2.
\end{align}
We now insert new variables $s_1, s_2, s_3, s_4\in A+A$ in (\ref{Nsystem1}) and (\ref{Nsystem2}) as follows:
$$ x_1+z_1 = s_1 = x'_1+z'_1,$$ $$ x_2+z_2 = s_2 = x'_2+z'_2,$$ $$ y_1+t_1 = s_3 = y'_1+t'_1,$$  $$ y_2+t_2 = s_4 = y'_2+t'_2.$$ 
Set $A_{s_i}=A\cap (s_i-A)$, then we have $x_1, x_1'\in A_{s_1}$, $x_2, x_2'\in A_{s_1}$, 
$y_1, y_1'\in A_{s_3}$, $y_2, y_2'\in A_{s_4}$.

In this setting, we have 
\[N=\sum_{s_1, s_2, s_3, s_4\in A+A}Q(A_{s_1}, A_{s_2}, A_{s_3}, A_{s_4}).\]
Applying Theorem \ref{thm-long}, and assuming the second and third terms are larger than the first term,  we have
\begin{align*}
N &\lesssim \sum_{s_1, s_2, s_3, s_4}|A_{s_1}|^{1/2}|A_{s_2}||A_{s_3}|^2|A_{s_4}|^{3/2}E^\times (A_{s_1}, A_{s_2})^{1/2}+\sum_{s_1, s_2, s_3, s_4}|A_{s_1}||A_{s_2}||A_{s_3}||A_{s_4}|^{3}\\&+\sum_{s_1, s_2, s_3, s_4}|A_{s_1}|^3|A_{s_2}||A_{s_3}||A_{s_4}|\\
&\lesssim  \left(\sum_{s_1, s_2}|A_{s_1}|^{1/2}|A_{s_2}|E^\times(A_{s_1}, A_{s_2})^{1/2}\right)\left(\sum_{s_3}|A_{s_3}|^2\right)\left(\sum_{s_4}|A_{s_4}|^{3/2}\right)+|A|^{10}\\
&\le |A|E^+(A)^{3/2} \left(\sum_{s_1, s_2}|A_{s_1}|^{1/2}|A_{s_2}|E^\times(A_{s_1}, A_{s_2})^{1/2}\right)+|A|^{10},
\end{align*}
where we have used the fact that 
\[\sum_{s}|A_s|^{3/2}\le \left(\sum_{s}|A_s|\right)^{1/2}\left(\sum_{s}|A_s|^2\right)^{1/2}, ~\sum_{s}|A_s|^2\le E^+(A).\]

Moreover,  using the fact $E^\times (A, B)\le |A|^2|B|$, we have
\begin{align}
&\sum_{s_1, s_2}|A_{s_1}|^{1/2}|A_{s_2}|E^\times(A_{s_1}, A_{s_2})^{1/2}\\
&\le \sum_{s_1, s_2}|A_{s_1}|^{3/2}|A_{s_2}|^{3/2}\le |A|^2 E^+(A, A).
\end{align}
In other words, we have proved that 
\[N\lesssim |A|^3E^+(A)^{5/2}+|A|^{10}.\]
If $N\lesssim |A|^{10}$, then the theorem follows from \eqref{FormulaA2A20}. Therefore, we can assume that 
\[N\lesssim |A|^3E^+(A)^{5/2}.\]

Let $\epsilon$ be a parameter chosen later. We now consider two cases: 
\begin{enumerate}
\item Suppose that $E^+(A) < |A|^{3-\epsilon}.$ Then we have 
\[N\le |A|^{\frac{9}{2}+6-\frac{5\epsilon}{2}}.\]
From \eqref{FormulaA2A20}, this implies that
\begin{equation}\label{chose1}|[A^2, A^2, 0]^2|\ge |A|^{\frac{11}{2}+\frac{5\epsilon}{2}}.\end{equation}
\item Suppose that $E^+(A) \geq |A|^{3-\epsilon}.$ Then we can write $E^+(A)=|A|^{3-\epsilon'}$ for some $\epsilon'< \epsilon< 1.$  Notice that  Theorem \ref{bsg} implies that there exists a subset $A'\subset A$ such that $|A'|\gg |A|^{1-\epsilon}$ and 
\[|A'-A'|\ll|A|^{4\epsilon} |A'|\ll |A'|^{1+\frac{4\epsilon}{1-\epsilon}}.\]
Since $|A|\le p^{9/16}$ by our assumption, using Lemma \ref{m-dif2} with the above inequality gives 
\begin{equation}\label{lessAP}|A'\cdot A'|\gtrsim |A'|^{\frac{14}{9}-\frac{8\epsilon}{1-\epsilon}}.\end{equation}
Moreover, one can easily check that 
\[|[A^2, A^2, 0]^2|\ge |A|^4\left\vert\left\lbrace x_1t_1+x_2t_2\colon x_1, t_1, x_2, t_2\in A\right\rbrace\right\vert.\]
Therefore, 
\[|[A^2, A^2, 0]^2|\ge |A|^4\left\vert\left\lbrace x_1t_1+x_2t_2\colon x_1, t_1\in A, x_2, t_2\in A'\right\rbrace\right\vert.\]
Moreover, Lemma \ref{m-dif1} gives us 
\begin{align*}
\left\vert\left\lbrace x_1t_1+x_2t_2\colon x_1, t_1\in A, x_2, t_2\in A'\right\rbrace\right\vert &\gg \min\left\lbrace |A||A'\cdot A'|^{1/2},~ p\right\rbrace \\
&\gtrsim \min \left\lbrace |A|^{1+(1-\epsilon)(\frac{7}{9}-\frac{4\epsilon}{1-\epsilon})}, ~p\right\rbrace,
\end{align*}
where  we also utilized the inequality \eqref{lessAP} and the fact that $|A'|\gg |A|^{1-\epsilon}.$
\end{enumerate}
Thus we obtain that if $|A|\le p^{9/16},$ then
\begin{align*}|[A^2, A^2, 0]^2|&\gtrsim \min  \left\lbrace|A|^{5+(1-\epsilon)(\frac{7}{9}-\frac{4\epsilon}{1-\epsilon})}, ~p|A|^4\right\rbrace \\
&\ge |A|^{5+(1-\epsilon)(\frac{7}{9}-\frac{4\epsilon}{1-\epsilon})}\end{align*}
provided that 
$ |A|^{1+(1-\epsilon)(\frac{7}{9}-\frac{4\epsilon}{1-\epsilon})} \le p.$
It is clear that with $\epsilon=5/131$, this estimate is satisfied since $|A|\le p^{1/2}$. We also obtain 
$$ |[A^2, A^2, 0]^2|\gtrsim |A|^{\frac{11}{2}+\frac{25}{262}}.$$ 
Finally, suppose the first term is the largest term. Then we have $N \leq \frac{E^+(A,A)^4}{p}$ which is smaller than $\frac{|A|^{12}}{p}$. Thus \eqref{FormulaA2A20} shows that 
$$ |[A^2, A^2, 0]^2|\gtrsim p|A|^{4} \geq |A|^6.$$ 
This is more than what we want and completes the proof of the theorem.
\end{proof}
\bigskip
Over finite arbitrary fields $\mathbb{F}_q$, we have the following result. 
\bigskip
\begin{theorem}\label{thm-cai}
Let $q = p^n$ and let $A$ be a subset of $\mathbb{F}^{*}_{q}$. If $|A \cap \lambda F| \leq |F|^{1/2}$ for any proper subfield $F$ of $\mathbb{F}_{q}$ and any $\lambda \in \mathbb{F}_q$, then we have 
\[|[A, A, 0]^2|\gtrsim |A|^{3+\frac{1}{11}},\]
and 
\[|[A^2, A^2, 0]^2|\gtrsim |A|^{5+\frac{1}{5}}.\]
\end{theorem}
\begin{proof}
We first observe that 
\[|[A, A, 0]^2|\ge |A|^2\cdot \max\left\lbrace |A+A|, |A\cdot A|\right\rbrace,\]
and 
\[|[A^2, A^2, 0]^2|\ge |A|^4\cdot |AA+AA|.\]
It follows from Theorem 2.7 and 2.8 that 
\[|AA+AA| \gg |A|^{6/5}, \max \left\lbrace |A+A|, |A\cdot A|\right\rbrace \gg |A|^{12/11}.\]
Therefore, we obtain 
\[|[A, A, 0]^2|\gtrsim |A|^{3+\frac{1}{11}},\]
and 
\[|[A^2, A^2, 0]^2|\gtrsim |A|^{5+\frac{1}{5}},\]
which completes the proof of the theorem.
\end{proof}
\bigskip
\begin{proof}[Proof of Theorem \ref{thm4}]
One can observe that
\begin{equation}\label{startingk}|[A^2, A^2, A]^2| \ge |A|^4|AA+AA+A+A|.\end{equation}
We now prove that if $|A|\le p^{9/16},$ then
\[|AA+AA+A+A|\gg |A|^{\frac{79}{45}}.\]
Indeed, we first prove that if $|A|\le p^{9/16},$
\begin{equation}\label{sumproductk}|AA+A+A|\gtrsim |A|^{\frac{3}{2}+\frac{1}{90}}.\end{equation}
To prove this inequality, we consider the following cases:
\begin{enumerate}
\item If $|A+A|\ge |A|^{1+\epsilon}$, then it follows from Lemma \ref{m-dif1} that 
\begin{equation}\label{con1k}|AA+A+A|\ge \min\left\{ |A|^{\frac{3}{2}+\frac{\epsilon}{2}},~p\right\}=|A|^{\frac{3}{2}+\frac{\epsilon}{2}},\end{equation} whenever 
$$|A|\le p^\frac{2}{3+\epsilon}.$$
\item If $|A+A|\le |A|^{1+\epsilon}$, then Lemma \ref{m-dif2} gives us that $|AA|\gtrsim |A|^{\frac{14}{9}-2\epsilon}$ under the condition $|A|\le p^{9/16}$. Hence, if $|A|\le p^{9/16},$ then 
\begin{equation}\label{con2k}|AA+A+A|\ge |AA|\gtrsim |A|^{\frac{14}{9}-2\epsilon}.\end{equation}
\end{enumerate}
Choosing $\epsilon=1/45$, we see from \eqref{con1k} and \eqref{con2k} that if $|A|\le p^{45/68}$ and $|A|\le p^{9/16}$, then
\[|AA+A+A|\gtrsim |A|^{\frac{3}{2}+\frac{1}{90}}.\]
Since $p^{45/68}\ge p^{9/16},$  we establish the inequality \eqref{sumproductk}.

By Lemma \ref{m-dif1} and the inequality \eqref{sumproductk}, we see that if $|A|\le p^{9/16}$, then
\[|AA+(AA+A+A)|\gg \min\left\{ |A| |AA+A+A|^{\frac{1}{2}},~p\right\} 
\gtrsim \min\left\{|A|^{\frac{7}{4}+\frac{1}{180}},~p\right\} =|A|^{\frac{79}{45}}.
\]
Finally, combining \eqref{startingk} and this estimate, we conclude that if $|A|\le p^{9/16}$, then
\[|[A^2, A^2, A]^2| \gtrsim |A|^{\frac{11}{2}+\frac{23}{90}},\]
which completes the proof of Theorem \ref{thm4} .
\end{proof}

\section*{Acknowledgments}
The authors would like to deeply thank Dr. Oliver Roche-Newton and Dr. Ilya Shkredov for many helpful discussions that make nice improvement for our Theorem \ref{thm3}.

D. Koh was supported by Basic Science Research Program through the National
Research Foundation of Korea(NRF) funded by the Ministry of Education, Science
and Technology(NRF-2015R1A1A1A05001374). T. Pham was supported by Swiss National Science Foundation grant P2ELP2175050. C-Y Shen was supported in part by MOST, through grant 104-2628-M-002-015 -MY4.

\end{document}